\documentclass[12pt]{amsart}
\usepackage{fourier}
\usepackage[T1]{fontenc}
\usepackage{graphics}
\usepackage{amsfonts,amssymb,color}
\usepackage[mathscr]{eucal}
\usepackage{amsmath, amsthm}
\usepackage{mathrsfs}
\usepackage{amsbsy}
\usepackage{dsfont}
\usepackage{bbm}
\usepackage{wasysym}
\usepackage{stmaryrd}
\usepackage{url}

\input xypic
\xyoption{all}

\makeindex
\makeglossary

\begin{document}
\baselineskip = 16pt

\newcommand \ZZ {{\mathbb Z}}
\newcommand \NN {{\mathbb N}}
\newcommand \RR {{\mathbb R}}
\newcommand \PR {{\mathbb P}}
\newcommand \AF {{\mathbb A}}
\newcommand \GG {{\mathbb G}}
\newcommand \QQ {{\mathbb Q}}
\newcommand \CC {{\mathbb C}}
\newcommand \bcA {{\mathscr A}}
\newcommand \bcC {{\mathscr C}}
\newcommand \bcD {{\mathscr D}}
\newcommand \bcF {{\mathscr F}}
\newcommand \bcG {{\mathscr G}}
\newcommand \bcH {{\mathscr H}}
\newcommand \bcM {{\mathscr M}}
\newcommand \bcI {{\mathscr I}}
\newcommand \bcJ {{\mathscr J}}
\newcommand \bcK {{\mathscr K}}
\newcommand \bcL {{\mathscr L}}
\newcommand \bcO {{\mathscr O}}
\newcommand \bcP {{\mathscr P}}
\newcommand \bcQ {{\mathscr Q}}
\newcommand \bcR {{\mathscr R}}
\newcommand \bcS {{\mathscr S}}
\newcommand \bcV {{\mathscr V}}
\newcommand \bcU {{\mathscr U}}
\newcommand \bcW {{\mathscr W}}
\newcommand \bcX {{\mathscr X}}
\newcommand \bcY {{\mathscr Y}}
\newcommand \bcZ {{\mathscr Z}}
\newcommand \goa {{\mathfrak a}}
\newcommand \gob {{\mathfrak b}}
\newcommand \goc {{\mathfrak c}}
\newcommand \gom {{\mathfrak m}}
\newcommand \gon {{\mathfrak n}}
\newcommand \gop {{\mathfrak p}}
\newcommand \goq {{\mathfrak q}}
\newcommand \goQ {{\mathfrak Q}}
\newcommand \goP {{\mathfrak P}}
\newcommand \goM {{\mathfrak M}}
\newcommand \goN {{\mathfrak N}}
\newcommand \uno {{\mathbbm 1}}
\newcommand \Le {{\mathbbm L}}
\newcommand \Spec {{\rm {Spec}}}
\newcommand \Gr {{\rm {Gr}}}
\newcommand \Pic {{\rm {Pic}}}
\newcommand \Jac {{{J}}}
\newcommand \Alb {{\rm {Alb}}}
\newcommand \Corr {{Corr}}
\newcommand \Chow {{\mathscr C}}
\newcommand \Sym {{\rm {Sym}}}
\newcommand \Prym {{\rm {Prym}}}
\newcommand \cha {{\rm {char}}}
\newcommand \eff {{\rm {eff}}}
\newcommand \tr {{\rm {tr}}}
\newcommand \Tr {{\rm {Tr}}}
\newcommand \pr {{\rm {pr}}}
\newcommand \ev {{\it {ev}}}
\newcommand \cl {{\rm {cl}}}
\newcommand \interior {{\rm {Int}}}
\newcommand \sep {{\rm {sep}}}
\newcommand \td {{\rm {tdeg}}}
\newcommand \alg {{\rm {alg}}}
\newcommand \im {{\rm im}}
\newcommand \gr {{\rm {gr}}}
\newcommand \op {{\rm op}}
\newcommand \Hom {{\rm Hom}}
\newcommand \Hilb {{\rm Hilb}}
\newcommand \Sch {{\mathscr S\! }{\it ch}}
\newcommand \cHilb {{\mathscr H\! }{\it ilb}}
\newcommand \cHom {{\mathscr H\! }{\it om}}
\newcommand \colim {{{\rm colim}\, }} 
\newcommand \End {{\rm {End}}}
\newcommand \coker {{\rm {coker}}}
\newcommand \id {{\rm {id}}}
\newcommand \van {{\rm {van}}}
\newcommand \spc {{\rm {sp}}}
\newcommand \Ob {{\rm Ob}}
\newcommand \Aut {{\rm Aut}}
\newcommand \cor {{\rm {cor}}}
\newcommand \Cor {{\it {Corr}}}
\newcommand \res {{\rm {res}}}
\newcommand \red {{\rm{red}}}
\newcommand \Gal {{\rm {Gal}}}
\newcommand \PGL {{\rm {PGL}}}
\newcommand \Bl {{\rm {Bl}}}
\newcommand \Sing {{\rm {Sing}}}
\newcommand \spn {{\rm {span}}}
\newcommand \Nm {{\rm {Nm}}}
\newcommand \inv {{\rm {inv}}}
\newcommand \codim {{\rm {codim}}}
\newcommand \Div{{\rm{Div}}}
\newcommand \CH{{\rm{CH}}}
\newcommand \sg {{\Sigma }}
\newcommand \DM {{\sf DM}}
\newcommand \Gm {{{\mathbb G}_{\rm m}}}
\newcommand \tame {\rm {tame }}
\newcommand \znak {{\natural }}
\newcommand \lra {\longrightarrow}
\newcommand \hra {\hookrightarrow}
\newcommand \rra {\rightrightarrows}
\newcommand \ord {{\rm {ord}}}
\newcommand \Rat {{\mathscr Rat}}
\newcommand \rd {{\rm {red}}}
\newcommand \bSpec {{\bf {Spec}}}
\newcommand \Proj {{\rm {Proj}}}
\newcommand \pdiv {{\rm {div}}}
\newcommand \wt {\widetilde }
\newcommand \ac {\acute }
\newcommand \ch {\check }
\newcommand \ol {\overline }
\newcommand \Th {\Theta}
\newcommand \cAb {{\mathscr A\! }{\it b}}

\newenvironment{pf}{\par\noindent{\em Proof}.}{\hfill\framebox(6,6)
\par\medskip}

\newtheorem{theorem}[subsection]{Theorem}
\newtheorem{conjecture}[subsection]{Conjecture}
\newtheorem{proposition}[subsection]{Proposition}
\newtheorem{lemma}[subsection]{Lemma}
\newtheorem{remark}[subsection]{Remark}
\newtheorem{remarks}[subsection]{Remarks}
\newtheorem{definition}[subsection]{Definition}
\newtheorem{corollary}[subsection]{Corollary}
\newtheorem{example}[subsection]{Example}
\newtheorem{examples}[subsection]{examples}

\title{Chow groups of  conic bundles in $\PR^5$ and the Generalised Bloch's conjecture}
\author{Kalyan Banerjee}

\address{Harish Chandra Research Institute, India}

\email{banerjeekalyan@hri.res.in}

\begin{abstract}
Consider the Fano surface of a conic bundle embedded in $\PR^5$. Let $i$ denote the natural involution acting on this surface. In this note we provide an obstruction to the identity action of the involution on the group of algebraically trivial zero cycles modulo rational equivalence on the surface.
\end{abstract}
\maketitle

\section{Introduction}

One of the very important problems in algebraic geometry is to understand the Chow group of zero cycles on a smooth projective surface with geometric genus and irregularity equal to $0$. It was already proved by Mumford \cite{M}, that for a smooth, projective complex surface of geometric genus greater than zero, the Chow group of zero cycles is infinite dimensional, in the sense that, it cannot be "parametrized" by an algebraic variety. The conjecture due to Spencer Bloch asserts the converse, that is, for a surface of geometric genus and irregularity zero, the Chow group of zero cycles is isomorphic to the group of integers. The Bloch's conjecture has been studied and proved in the case when the surface is not of general type by \cite{BKL} and for surfaces of general type by \cite{B}, \cite{IM}, \cite{GP}, \cite{PW}, \cite{V},\cite{VC}. Inspired by the Bloch's conjecture, the following conjecture is made, which is a generalisation \cite{Vo}[conjecture 11.19].

\medskip

\textit{ Conjecture : Let $S$ be a smooth projective surface over the field of complex numbers and let $\Gamma$ be a codimension two cycle on $S\times S$. Suppose that $\Gamma^*$ acts as zero on the space of globally holomorphic two forms on $S$, then $\Gamma_*$ acts as zero on the kernel of the albanese map from $\CH_0(S)$ to $Alb(S)$.}

\medskip

This conjecture was studied in detail when the correspondence $\Gamma$ is the $\Delta- Graph(i)$, where $i$ is a sympletic involution on a $K3$ surface by \cite{GT}, \cite{HK},\cite{Voi}. In the example of K3 surfaces the push-forward induced by the involution acts as identity on Chow group of zero cycles of degree zero.

Inspired by this conjecture we consider the following question in this article. Let $X$ be a smooth, cubic fourfold in $\PR^5$. Consider a line $l$ in $\PR^5$, embedded in $X$. Considering the projection from the line $l$ to $\PR^3$, we have a conic bundle structure on the cubic $X$. Let $S$ be the discriminant surface of this conic bundle. Let $T$ be the double cover of $S$ inside the Fano variety of lines $F(X)$ of $X$, arising from the conic bundle structure. Then $T$ has a natural involution and we observe that the group of algebraically trivial zero cycles on $T$ modulo rational equivalence (denoted by $A_0(T)$) maps surjectively onto the algebraically trivial one cycles on $X$ modulo rational equivalence (denoted by $A_1(X)$).  The  action of the involution has as its invariant part equal to the $A_0(S)$ and as anti-invariant part equal to $A_1(X)$. The involution cannot act as $+1$ on the group $A_1(X)$, as it will follow that all the elements of $A_1(X)$ are $2$-torsion, hence $A_1(X)$ is weakly representable. This is not true by the main theorem of \cite{SC}. Now the question is, what is the obstruction to the $+1$ action of the involution in terms of the geometry of $S,T$.

\medskip

\begin{theorem}
\label{theorem3}
Let $S$ be the discriminant surface as mentioned above. Then for any very ample line bundle $L$ on $S$ we cannot have the equality
$$L^2-g+1=g+n$$
where $g$ is the genus of the curve in the linear system of $L$ and $n$ is a positive integer.
\end{theorem}

\medskip

This result motivates the following:

\begin{corollary}
Suppose that we have a surface of general type $S$ with geometric genus zero and we have an involution $i$ on the surface $S$ having only finitely many fixed points. Suppose that there exists a very  ample line bundle $L$, on the minimal desingularization of the quotient surface $S/i$ such that the following equality
$$L^2-2g+1=n$$
is true, here $g$ is the genus of the smooth curves in the linear system $|L|$ and $n$ is some positive integer. Then the involution $i_*$ acts as identity on the group $A_0(S)$.
\end{corollary}

For the proof of the above theorem and the corollary, we follow the approach of the proof for the example of K3 surfaces due to Voisin as in \cite{Voi}. The proof involves two steps. First is that we invoke the notion of finite dimensionality in the sense of Roitman as in \cite{R1} and prove that the finite dimensionality of the image of a homomorphism from $A_0(T)$ to $A_1(X)$ (respectively from $A_0(S)\to A_0(S)$) implies that the homomorphism factors through the albanese map $A_0(T)\to Alb(T)$ (or $A_0(S)\to Alb(S)$ respectively). The second step is to show that, if we have the equality as above \ref{theorem3}, then the image of the homomorphism induced by the difference of the diagonal and the graph of the involution from $A_0(T)$ to $A_1(X)$ (or $A_0(S)\to A_0(S)$) is finite dimensional, yielding the $+1$ action of the involution on $A_1(X)$ or $A_0(S)$ respectively.

As an implication of the above corollary we obtain the Bloch's conjecture for the Craighero-Gattazzo surface of general type with geometric genus zero, studied in \cite{CG},\cite{DW}. This class of surfaces, is obtained as minimal resolution of singularities of  singular quintics in $\PR^3$ invariant under an involution and having four isolated, simple elliptic singular points.

{\small \textbf{Acknowledgements:} The author would like to thank the hospitality of IISER-Mohali, for hosting this project. The author is indebted Kapil  Paranjape for some useful conversations relevant to the theme of the paper. The author likes to thank Claire Voisin for her  advice on the theme of the paper. The author is indebted to J.L.Colliot-Thelene, B.Poonen, and the anonymous referee for finding out a crucial mistake in the earlier version of the manuscript.}

\medskip

{\small Assumption: We work over the field of complex numbers.}

\medskip

\section{Finite dimensionality in the sense of Roitman and one-cycles on cubic fourfolds}

Let $P$ be a subgroup of the group of algebraically trivial one cycles modulo rational equivalence on a smooth projective fourfold $X$, the latter is denoted by $A_1(X)$. Following \cite{R1}, we say that the subgroup $P$ is finite dimensional, if there exists a smooth projective variety $W$, and a correspondence $\Gamma$ on $W\times X$, of correct codimension, such that $P$ is contained in the set $\Gamma_*(W)$.

Let $X$ be a cubic fourfold. Consider a line $l$ on $X$ and project from $l$ onto $\PR^3$. Consider the blow up of $X$ along $l$. Then the blow up $X_l$ has a conic bundle structure over $\PR^3$. Let $S$ be the surface in $\PR^3$ such that for any closed point on $S$, the inverse image is the union of two lines in $\PR^3$. 
Let $T$ be the variety in $F(X)$ which is the double cover of $S$. Precisely it means the following. Let us consider
$$\bcU:=\{(l',x):x\in l', \pi_l(x)\in S \}$$
inside $F(X)\times X$. Then its projection to $F(X)$ is $T$ and we have a 2:1 map from $T$ to $S$, which is branched along finitely many points. So $T$ is surface.

Now for a hyperplane section $X_t$, let $l_1,l_2$ be two lines contained in $X_t$. By general position argument these two lines can be disjoint from $l$ inside $X$ and they are contained in $\PR^2$, so under the projection from $l$ they are mapped to two rational curves in $\PR^2$. Thus by Bezout's theorem they must intersect at a point $z$, so the inverse image of $z$ under the projection are two given lines $l_1,l_2$, which tells us that the map from $A_0(T_t)$ to $A_1(X_t)$ is onto, here $T_t$ is the double cover (for a general $t$) of $S_t$, where $S_t$ is the discriminant curve of the projection $\pi_l: X_t\to \PR^2$. This in turn says that $A_0(T)$ to $A_1(X)$ is onto, because $A_1(X)$ is generated by $A_1(X_t)$, where $t$ varies.

\begin{theorem}
\label{theorem1}
Let $Z$ be a correspondence supported on $T\times X$. Suppose that the image of $Z_*$ from $A_0(T)$ to $A_1(X)$ is finite dimensional. Then $Z_*$ factors through the albanese map of $T$.
\end{theorem}

\begin{proof}
The proof of this theorem follows the approach of \cite{Voi}[Theorem 2.3]. Since $Z_*$ has finite dimensional image, there exists a smooth projective variety $W$ and a correspondence $\Gamma$ supported on $W\times X$ such that image of $Z_*$ is contained in $\Gamma_*(W)$. Let $C$ inside $T$ be a smooth, hyperplane section (after fixing an embedding of $T$ into a projective space). Then by Lefschetz theorem on hyperplane sections we have that $J(C)$ maps onto $Alb(T)$. So the kernel is an abelian variety, denoted by $K(C)$. First we prove the following.

\begin{lemma}
The abelian variety $K(C)$ is simple for a general hyperplane section $C$ of $T$.
\end{lemma}
\begin{proof}
The proof of this lemma follows the approach of \cite{Voi}[Proposition 2.4]. Let if possible there exists a non-trivial proper abelian subvariety $A$ inside $K(C)$. Now $K(C)$ corresponds to the Hodge structure
$$\ker(H^1(C,\QQ)\to H^3(T,\QQ))\;.$$
Let $T\to D$ be a Lefschetz pencil such that a smooth fiber is $C$. Then the fundamental group $\pi_1(D\setminus 0_1,\cdots,0_m,t)$ acts irreducibly on the Hodge structure mentioned above, \cite{Vo}[Theorem 3.27]. Here $t$ corresponds to the smooth fiber $C$. Now the abelian variety $A$ corresponds to a Hodge sub-structure $H$ inside the above mentioned Hodge structure. Let $A_D$ be the base change of $A$ over the spectrum of the function field $\CC(D)$. For convenience, let us continue to denote $A_D$ by $A$. Then consider a finite extension $L$ of $\CC(D)$ inside $\overline{\CC(D)}$, such that $A$, $K(C)$ are defined over $L$. Then we spread $A,K(C)$, over a Zariski open $U'$ in $D'$, where $\CC(D')=L$ and $D'$ is a smooth, projective curve which maps finitely onto $D$. Denote these spreads by $\bcA,\bcK$ over $U'$. By throwing out more points from $U'$ we get that $\bcA\to U', \bcK\to U'$ are fibrations, of the underlying smooth manifolds. So the fundamental group $\pi_1(U',t')$ acts on $H$, which is the $2d-1$-th cohomology of $A$ ($d=\dim(A)$), and on $\ker(H^1(C,\QQ)\to H^3(T,\QQ))$. Since $U'$ maps finitely onto a Zariski open $U$ of $D$, we have that $\pi_1(U',t')$ is a finite index subgroup of $\pi_1(U,t)$. Now it is a consequence of the Picard-Lefschetz formula that $H$ is a $\pi_1(U,t)$ stable subspace of $\ker(H^1(C,\QQ)\to H^3(T,\QQ))$. The latter is irreducible under the action of $\pi_1(U,t)$. So we get that $H$ is either zero or all of $\ker(H^1(C,\QQ)\to H^3(T,\QQ))$. Therefore by the equivalence of abelian varieties and weight one, polarized Hodge structures, $A$ is either zero or all of $K(C)$.

\end{proof}

Now consider sufficiently ample hyperplane sections of $T$, so that the dimension of $K(C)$ is arbitrarily large, and hence strictly greater than $\dim(W)$. Consider the subset $R$ of $K(C)\times W$, consisting of pairs $(k,w)$ such that
$$Z_*j_*(k)=\Gamma_*(w)$$
here $j: C\to T$ is the closed embedding of $C$ into $T$. Since the image of $Z_*$ is finite dimensional, the projection from $R$ onto $K(C)$ is surjective. By the Mumford-Roitman argument on Chow varieties \cite{R}, $R$ is a countable union of Zariski closed subsets in the product $K(C)\times W$. By the uncountability of the field of complex numbers it follows that some component $R_0$ of $R$, dominates $K(C)$. Therefore we have that
$$\dim(R_0)\geq \dim(K(C))>\dim (W)\;.$$
So the fibers of the map $R_0\to W$ are positive dimensional. Since the abelian variety $K(C)$ is simple, the fibers of $R_0\to W$ generate the  abelian variety $K(C)$. So for any zero cycle $z$ supported on the fibers of $R_0\to W$, we have that
$$Z_*j_*(z)=\deg(z)\Gamma_*(w)$$
since $z$ is of degree zero, it follows that $Z_*$ vanishes on the fibers of $R_0\to W$, which is positive dimensional, hence on all of $K(C)$, by the simplicity of $K(C)$.

Now to prove that the map $Z_*$ factors through $alb$, we consider a zero cycle $z$ of degree zero, which is given by a tuple of $2k$ points for a fixed positive integer $k$. Then we blow up $T$ along these points, denote the blow up by $\tau:T'\to T$. Let $E_i$'s be the exceptional divisor of the blow up, we choose $H$ in $\Pic(T)$, such that $L=\tau^*(H)-\sum_i E_i$ is ample (this can be obtained by Nakai Moisezhon-criterion for ampleness). Now consider a sufficiently large, very ample multiple of $L$, and apply the previous method to a general member $C'$ of the corresponding linear system. Then $K(C')$ is a simple abelian variety. Also $\tau(C')$ contains all the points at which we have blown up. Suppose that the corresponding cycle $z$ is annihilated by $alb_T$, then any of its lifts to $T'$ say $z'$, is annihilated by $alb_{T'}$ and is supported on $K(C')$. So applying the previous argument to the correspondence $Z'=Z\circ \tau$, we have that
$$Z_*(z)=Z'_*(z')=0\;.$$

\end{proof}

Let $i$ be the involution on $T$, then this involution induces an involution on $A_1(X)$. Consider the homomorphism given by the difference of identity and the induced involution on $A_1(X)$, call it $Z_{1*}$. It is clear from  \ref{theorem1} that the image of $Z_*Z_{1*}$ cannot be finite dimensional, otherwise the involution will act as $+1$ on $A_1(X)$, leading to the fact that $A_1(X)=\{0\}$. Now we prove the following:

\begin{theorem}
\label{theorem2}
Let $S$ be the discriminant surface, mentioned above. Then for any very ample line bundle $L$ on $S$ the equality
$$L^2-g+1=g+n$$
cannot hold, where $g$ is the genus of a curve in the complete linear system of $L$ and $n$ is a positive integer.
\end{theorem}

\begin{proof}
The proof of this theorem follows the approach of \cite{Voi}[Proposition 2.5]. The discriminant surface $S$ is a quintic, hence its irregularity is zero. Consider a very ample line bundle $L$ on the quintic $S$. Let $g$ be the genus of a smooth curve in the linear system $|L|$. Now we calculate the dimension of $|L|$. Consider the exact sequence
$$0\to \bcO(C)\to \bcO(S)\to \bcO(S)/\bcO(C)\to 0$$
tensoring with $\bcO(-C)$ we have
$$0\to \bcO(S)\to \bcO(-C)\to\bcO(-C)|_C\to 0\;.$$
Taking sheaf cohomology we have
$$0\to \CC\to H^0(S,L)\to H^0(C,L|_C)\to 0$$
since the irregularity of the surface is zero.
On the other hand by Nakai-Moisezhon criterion the intersection number $L|_C$ is positive, so $L$ restricted to $C$ has positive degree, by Riemann-Roch this implies that
$$\dim(H^0(C,L|_C))=L^2-g+1\;,$$
provided that we have the equality
$$L^2-g+1=g+n$$
for some positive integer $n$.
Then the linear system of $L$ is of dimension $g+n$. Now consider the smooth, projective curves $C$ in this linear system $|L|$ and their double covers $\wt{C}$ (this is actually a covering for a general $C$, as the map $T\to S$ is branched along a finite set of points). By Bertini's theorem a general $\wt{C}$ is smooth. By the Hodge index theorem it follows that, $\wt{c}$ is connected. If not, suppose that it has two components $C_1,C_2$. Since $C^2>0$, we have $C_i^2>0$ for $i=1,2$ and since $\wt{C}$ is smooth we have that $C_1.C_2=0$. Therefore the intersection form restricted to $\{C_1,C_2\}$ is semipositive. This can only happen when $C_1$, $C_2$ are proportional and $C_i^2=0$, for $i=1,2$, which is not possible.

Now let $(t_1,\cdots, t_{g+n})$ be a point on $T^{g+n}$, which gives rise to the tuple $(s_1,\cdots,s_{g+n})$ on $S^{g+n}$, under the quotient map. There exists a unique, smooth curve $C$ containing all these points (if the points are in general position). Let $\wt{C}$ be its double cover on $T$. Then $(t_1,\cdots,t_{g+n})$ belongs to $\wt{C}$. Consider the zero cycle
$$\sum_i t_i-\sum_i i_*(t_i)$$
this belongs to the image of  $P(\wt{C}/C)$ in $A_0(T)$, $P(\wt{C}/C)$ is the Prym variety corresponding to the double cover. So the image of
$$\sum_i \left(Z_*(t_i)-i_*Z_*(t_i)\right)$$
is an element in the image of this Prym variety under the homomorphism
$$A_0(T)\to A_1(X)\;.$$
So the map
$$T^{g+n}\to A_1(X)$$
given by
$$(t_1,\cdots,t_{g+n})\mapsto \sum_i Z_*(t_i)-i_*Z_*(t_i) $$
factors through the Prym fibration $\bcP(\wt {\bcC}/\bcC)$, given by
$$(t_1,\cdots,t_{g+n})\mapsto alb_{\wt{C}}\left(\sum_i t_i-i(t_i)\right)$$
here $\bcC, \wt{\bcC}$ are the universal smooth curve and the universal double cover of $\bcC$ over $|L|_0$ parametrizing the smooth curves in  the linear system $|L|$. By dimension count, the dimension of $\bcP(\wt {\bcC}/\bcC)$ is $2g+n-1$. On the other hand we have that dimension of $T^{g+n}$ is $2g+2n$. So the map
$$T^{g+n}\to \bcP(\wt {\bcC}/\bcC)$$
has positive dimensional fibers, and hence the map
$$T^{g+n}\to A_1(X)$$
has positive dimensional fibers. So the general fiber of $$T^{g+n}\to A_1(X)$$ contains a curve. Let $H$ be the hyperplane bundle pulled back onto the quintic surface $S$. It is very ample. Pull it back further onto $T$, to get an ample line bundle on $T$. Call it $L'$. Then the divisor $\sum_i \pi_i^{-1}(L')$ is ample on $T^{g+n}$, where $\pi_i$ is the $i$-th co-ordinate projection from $T^{g+n}$ to $T$. Therefore the curves in the fibers of the above map intersect the divisor $\sum_i \pi_i^{-1}(L')$.
So we get that there exist points in $F_s$ (the general fiber over a cycle $s$ in $A_1(X)$) contained in $C\times T^{g+n-1}$ where $C$ is in the linear system of $L'$. Then consider the elements of $F_s$ the form $(c,s_1,\cdots,s_{g+n-1})$, where $c$ belongs to $C$. Considering the map from $T^{g+n-1}$ to $A_1(X)$ given by
$$(s_1,\cdots,s_{g+n-1})\mapsto Z_*(\sum_i s_i+c-\sum_i i(s_i)-i(c))\;,$$
we see that this map factors through the Prym fibration and the map from $T^{g+n-1}$ to $\bcP(\wt{\bcC}/\bcC)$ has positive dimensional fibers, since $n$ is large. So it means that, if we consider an element $(c,s_1,\cdots,s_{g+n-1})$ in $F_s$ and a curve through it, then it intersects the ample divisor given by $\sum_i \pi_i^{-1}(L')$, on $T^{g+n-1}$. Then we have some of $s_i$ is contained in $C$. So iterating this process we get that elements of $F_s$ are supported on $C^k\times T^{g+n-k}$, where $k$ is some natural number depending on $n$. Note that the genus of $C$ is fixed and equal to $11$ and less than $k$ and for a choice of a large multiple of the very ample line bundle $L$. Thus the elements of $F_s$ are supported on $C^{n_0}\times T^{g+n-k}$.
Therefore considering $\Gamma=Z_1\circ Z$, we get that $\Gamma_*(T^{g+n})=\Gamma_*(T^{m_0})$, where $m_0$ is strictly less than $g+n$.

Now we prove by induction that $\Gamma_*(T^{m_0})=\Gamma_*(T^m)$ for all $m\geq g+n$.
So suppose that $\Gamma_*(T^k)=\Gamma^*(T^{m_0})$ for $k\geq g+n$, then we have to prove that $\Gamma_*(T^{k+1})=\Gamma_*(T^{m_0})$. So any element in $\Gamma_*(T^{k+1})$ can be written as  $\Gamma_*(t_1+\cdots+t_{m_0})+\Gamma_*(t)$. Now let $k-m_0=m$, then $m_0+1=k-m+1$. Since $k-m<k$, we have $k-m+1\leq k$, so $m_0+1\leq k$, so we have the cycle
$$\Gamma_*(t_1+\cdots+t_{m_0})+\Gamma_*(t)$$
supported on $T^k$, hence on $T^{m_0}$. So we have that $\Gamma_*(T^{m_0})=\Gamma_*(T^k)$ for all $k$ greater or equal than $g+n$. Now any element $z$ in $A_0(T)$, can be written as a difference of two effective cycle $z^+,z^-$ of the same degree. Then we have
$$\Gamma_*(z)=\Gamma_*(z^+)-\Gamma_*(z_-)$$
and $\Gamma_(z_{\pm})$ belong to $\Gamma_*(T^{m_0})$. So let $\Gamma'$ be the correspondence on $T^{2m_0}\times T$ defined as
$$\sum_{l\leq m_0}(pr_{l},pr_T)^*\Gamma-\sum_{m_0+1\leq l\leq 2m_0}(pr_l,pr_T)^* \Gamma$$
where $\pr_l$ is the $l$-th projection from $T^l$ to $T$, and $\pr_T$ is from $T^{2m_0}\times T$ to the last copy of $T$. Then we have
$$\im(\Gamma_*)=\Gamma'_*(T^{2m_0})\;.$$
This would imply that the image of $\Gamma_*$ is finite dimensional, so by \ref{theorem1} we have that the induced involution on $A_1(X)$ acts as identity. The involution acts as $-\id$ on $A_1(X)$. Hence all elements of $A_1(X)$ is a $2$-torsion. This will lead to a contradiction to the fact that $A_1(X)$ is infinite dimensional \cite{SC}.
\end{proof}

Now we proceed to the proof of the corollary stated in the introduction regarding the generalised Bloch conjecture on surfaces of general type with geometric genus zero and with an involution $i$. The result is as follows:

\begin{corollary}
\label{cor1}
Suppose that we have a surface of general type $S$ with geometric genus zero and we have an involution $i$ on the surface $S$ having only finitely many fixed points. Suppose that there exists a very  ample line bundle $L$, on the minimal desingularization of the quotient surface $S/i$ (by the involution) such that the following equality
$$L^2-2g+1=n$$
is true, here $g$ is the genus of the smooth, projective curves in the linear system $|L|$, and $n$ is some positive integer. Then the involution $i_*$ acts as identity on the group $A_0(S)$.
\end{corollary}

\begin{proof}
Consider the resolution of singularity of the  surface $S/i$. It is the quotient by the involution acting on the surface $\wt{S}$, obtained by blowing up the isolated fixed points of $i$ acting on $S$. Call this quotient $\wt{S}/i$. Since it is dominated by a surface of irregularity zero (namely $\wt{S}$), it has irregularity zero. Consider a very ample line bundle $L$ on $\wt{S}/i$. Let $g$ be the genus of a smooth, projective curve in the linear system $|L|$. Now we calculate the dimension of $|L|$. Consider the exact sequence
$$0\to \bcO(C)\to \bcO(\wt{S}/i)\to \bcO(\wt{S}/i)/\bcO(C)\to 0$$
tensoring with $\bcO(-C)$ we get
$$0\to \bcO(\wt{S}/i)\to \bcO(-C)\to\bcO(-C)|_C\to 0\;.$$
Taking sheaf cohomology we get
$$0\to \CC\to H^0(\wt{S}/i,L)\to H^0(C,L|_C)\to 0$$
since the irregularity of the surface $\wt{S}/i$ is zero.
On  the other hand by Nakai-Moiseshon criterion the intersection number $L|_C$ is positive, so $L$ restricted to $C$ has positive degree, by Riemann-Roch this implies
$$\dim(H^0(C,L|_C))=L^2-g+1\;, $$
provided that we have the equality
$$L^2-g+1=g+n$$
for some positive integer $n$.
Then the linear system of $L$ is of dimension $g+n$. Now consider a smooth, projective curves $C$ in this linear system $|L|$ and its branched double cover $\wt{C}$, branched along the intersection of $\wt{C}$ with $E_i$, where $E_i$'s are the exceptional curves arising from the blow up $\wt{S}\to S$. By Bertini's theorem a general $\wt{C}$ is smooth. By the Hodge index theorem it follows that, it is connected. If not, suppose that it has two components $C_1,C_2$. Since $C^2>0$, we have $C_i^2>0$ for $i=1,2$ and since $\wt{C}$ is smooth we have that $C_1.C_2=0$. Therefore the intersection form restricted to $\{C_1,C_2\}$ is semipositive. This can only happen when $C_1$, $C_2$ are proportional and $C_i^2=0$, for $i=1,2$, which is not possible as $C_1+C_2$ is ample on $\wt{S}$.

Now let $(t_1,\cdots, t_{g+n})$ be a point on $\wt{S}^{g+n}$, which gives rise to the tuple $(s_1,\cdots,s_{g+n})$ on $(\wt{S}/i)^{g+n}$, under the quotient map. There exists a unique, smooth curve $C$ containing all these points (if the points are in general position). Let $\wt{C}$ be its branched double cover of $C$ in $\wt{S}$. Then $(t_1,\cdots,t_{g+n})$ belongs to $\wt{C}$. Consider the zero cycle
$$\sum_i t_i-\sum_i i_*(t_i)$$
this belongs to $P(\wt{C}/C)$, which is the Prym variety corresponding to the double cover $\wt{C}\to C$. So the image of
$$\sum_i \left(t_i-i_*(t_i)\right)$$
under the push-forward $j_{\wt{C}*}$
is an element in the image  under the homomorphism
$$\id-i_*: A_0(\wt{S})\to A_0(\wt{S})$$
 So the map
$$\wt{S}^{g+n}\to A_0(\wt{S})$$
given by
$$(t_1,\cdots,t_{g+n})\mapsto \sum_i (t_i-i_*(t_i)) $$
factors through the Prym fibration $\bcP(\wt {\bcC}/\bcC)$, given by
$$(t_1,\cdots,t_{g+n})\mapsto alb_{\wt{C}}\left(\sum_i t_i-i(t_i)\right)$$
here $\bcC, \wt{\bcC}$ are the universal family of smooth curves in $|L|$ and the universal double cover of $\bcC$ respectively, over $|L|_0$ parametrizing the smooth curves in the linear system $|L|$. By dimension count the dimension of $\bcP(\wt {\bcC}/\bcC)$ is $2g+n-1+m/2$, where $m$ is the number of branch points on the curve $\wt{C}$ counted with multiplicities. On the other hand we have that dimension of ${\wt{S}}^{g+n}$ is $2g+2n$. So the map
$${\wt{S}}^{g+n}\to \bcP(\wt {\bcC}/\bcC)$$
has fiber dimension equal to
$$2g+2n-2g-n+1-m/2=n+1-m/2\;.$$
Considering a large multiple of the very ample line bundle $L$, we can assume that the above number is positive. Indeed we have
$$L^2-2g+1=-L.K_{\wt{S}/i}-1=n>0$$
and $$K_{\wt{S/i}}=f^*(K_{S/i})+E$$
where $E$ is the exceptional divisor, $f$ is the regular map from $\wt{S}/i$ to $S/i$. Here we consider $\wt{S}$ is the blow up of $S$ along the unique fixed point of $i$. The calculation for finitely many fixed points greater than one is similar. Let $L$ be equal to $m'f^*(H)-m'E$ which is very ample, where $H$ is a very ample line bundle on $S/i$, after fixing an embedding into some projective space. Then we have to prove that

$$L.(-2K_{\wt{S}/i}-E)-2>0$$

that is

$$L.(-2f^*(K_{S/i})-3E)-2>0$$

putting the expression of $L$, the condition to be proven is

$$-(m'(f^*(H)-E)(2f^*(K_{S/i})+3E))-2=-2m'f^*(H).f^*(K_{S/i})-3m'-2>0$$

But by the adjunction formula on $\wt{S}/i$ we have

$$L^2-2g+1=-L.K_{\wt{S}/i}-1$$

on the other hand

$$L^2-2g+1>0$$

by the assumption of the theorem.
Therefore

$$-m'f^*(H).f^*(K_{S/i})-m'-1=-m'f^*(H).f^*(K_{S/i})-(m'+1)>0$$

so

$$-2m'f^*(H).f^*(K_{S/i})> 2m'+2\;.$$

Therefore choosing $l>3$, such that $m'f^*(lH)-m'E$ is very ample, we have
$$-2f^*(lm'H).f^*(K_{S/i})> 2l(m'+1)>3m'+2$$
for large values of $l$.
Also note that for $L=m'(f^*(lH)-E)$, $l>1$ we have
$$L^2-2g+1=-L.K_{\wt{S/i}}-1=-f^*(m'lH).f^*(K_{S/i})-m'-1=-m'lf^*(H).f^*(K_{S/i})-m'-1$$
we know that
$$-m'f^*(H).f^*(K_{S/i})>m'+1$$
so
$$-m'lf^*(H).f^*(K_{S/i})-m'-1>(m'l-1)(m'+1)> 0\;.$$
So for $L=f^*(m'lH)-m'E$ we have the equality
$$L^2-2g+1=n$$
for some positive integer $n$.

 So the fiber contains a curve. Let $H$ be the hyperplane bundle pulled back onto the surface $\wt{S}/i$, after fixing an embedding of $\wt{S}/i$ into some projective space. It is very ample. Pull it back further onto $\wt{S}$, to get an ample line bundle on $\wt{S}$. Call it $L'$. Then the divisor $\sum_i \pi_i^{-1}(L')$ is ample on $\wt{S}^{g+n}$, where $\pi_i$ is the $i$-th co-ordinate projection from $\wt{S}^{g+n}$ to $\wt{S}$. Therefore the curves in the fibers of the above map intersect the divisor $\sum_i \pi_i^{-1}(L')$.
So there exist points in $F_s$ (the general fiber of $\wt{S}^{g+n}\to A_0(\wt{S})$ over a cycle $s$ in $A_0(\wt{S})$) contained in $C\times \wt{S}^{g+n-1}$ where $C$ is in the linear system of $L'$. Then consider the elements of $F_s$ the form $(c,s_1,\cdots,s_{g+n-1})$, where $c$ belongs to $C$. Considering the map from $\wt{S}^{g+n-1}$ to $A_0(\wt{S})$ given by
$$(s_1,\cdots,s_{g+n-1})\mapsto (\sum_i s_i+c-\sum_i i_*(s_i)-i_*(c))\;,$$
we see that this map factors through the Prym fibration and the map from $\wt{S}^{g+n-1}$ to $\bcP(\wt{\bcC}/\bcC)$ has positive dimensional fibers, by choosing $l$ and hence $n$ to be large. So, if we consider an element $(c,s_1,\cdots,s_{g+n-1})$ in $F_s$ and a curve through it, then it intersects the ample divisor given by $\sum_i \pi_i^{-1}(L')$, on $\wt{S}^{g+n-1}$. Then we have some of $s_i$ is contained in $C$. So iterating this process we have, the elements of $F_s$ are supported on $C^k\times \wt{S}^{g+n-k}$, where $k$ is some natural number depending on $n$. Note that the genus of $C$ is fixed and it is less than $k$ for a choice of a very large multiple of the very ample line bundle $L$. Thus the elements of $F_s$ are supported on $C^{n_0}\times \wt{S}^{g+n-k}$.
Therefore considering $\Gamma=\Delta_{\wt{S}}-Gr(i)$, we get that $\Gamma_*(\wt{S}^{g+n})=\Gamma_*(\wt{S}^{m_0})$, where $m_0$ is strictly less than $g+n$.

Now we prove by induction that $\Gamma_*(\wt{S}^{m_0})=\Gamma_*(\wt{S}^m)$ for all $m\geq g+n$.
So suppose that $\Gamma_*(\wt{S}^k)=\Gamma^*(\wt{S}^{m_0})$ for $k\geq g+n$, then we have to prove that $\Gamma_*(\wt{S}^{k+1})=\Gamma_*(\wt{S}^{m_0})$. So any element in $\Gamma_*(\wt{S}^{k+1})$ can be written as  $$\Gamma_*(t_1+\cdots+t_{m_0})+\Gamma_*(t)$$ Now let $k-m_0=m$, then $m_0+1=k-m+1$. Since $k-m<k$, we have $k-m+1\leq k$, so $m_0+1\leq k$, so we have the cycle
$$\Gamma_*(t_1+\cdots+t_{m_0})+\Gamma_*(t)$$
supported on $\wt{S}^k$, hence on $\wt{S}^{m_0}$. So we have $$\Gamma_*(\wt{S}^{m_0})=\Gamma_*(\wt{S}^k)$$ for all $k$ greater or equal than $g+n$. Now any element $z$ in $A_0(\wt{S})$, can be written as a difference of two effective cycles $z^+,z^-$ of the same degree. Then we have
$$\Gamma_*(z)=\Gamma_*(z^+)-\Gamma_*(z_-)$$
and $\Gamma_(z_{\pm})$ belong to $\Gamma_*(\wt{S}^{m_0})$. So let $\Gamma'$ be the correspondence on $\wt{S}^{2m_0}\times \wt{S}$ defined as
$$\sum_{l\leq m_0}(pr_{l},pr_{\wt{S}})^*\Gamma-\sum_{m_0+1\leq l\leq 2m_0}(pr_l,pr_{\wt{S}})^* \Gamma$$
where $\pr_l$ is the $l$-th projection from $\wt{S}^l$ to $\wt{S}$, and $\pr_{\wt{S}}$ is from ${\wt{S}}^{2m_0}\times \wt{S}$ to the last copy of $\wt{S}$. Then we have
$$\im(\Gamma_*)=\Gamma'_*(\wt{S}^{2m_0})\;.$$
This would imply that the image of $\Gamma_*$ is finite dimensional, so as proved in \cite{Voi}[Theorem 2.3] the induced involution on $A_0(\wt{S})$ factors through the Albanese variety of $\wt{S}$ which is trivial. Hence $i_*$ acts as identity on $A_0(\wt{S})$. By the blow up formula
$$A_0(\wt{S})\cong A_0(S)$$
hence the involution $i_*$ acts as identity on $A_0(S)$.
\end{proof}

\begin{remark}
\label{rem1}
Suppose in the above corollary \ref{cor1} we have the fixed locus of the involution consisting of finitely many isolated fixed points and one rational curve. Then on $\wt{S}/i$ we have to prove that the number
$$L.(-2K_{\wt{S}/i}-\sum_j E_j-R)-2>0$$
Here $R$ is the strict transform of the rational curve component in the fixed locus, $E_j$ is the exceptional curve over the isolated fixed point $p_j$.
Putting $$L=m(f^*(H)-\sum_j E_j)$$
we have to prove that
$$-m(f^*(H)-\sum_j E_j)(2f^*(K_{S/i})+3\sum_j E_j+R)-2>0$$
So for simplicity let us assume that the number of isolated fixed point is one, so there is one exceptional divisor.
Thus we have to prove that
$$-m(f^*(H)- E)(2f^*(K_{S/i})+3E+R)-2>0$$
that is
$$-2mf^*(H)f^*(K_{S/i})-3m-2-mf^*(H).R>0$$
Since $R=f^*(L)$ where $L$ is a line in $S/i$, we have
$$f^*(H).f^*(L)=f^*(H.L)=f^*(p)=2p$$
Putting this in the above equation
$$-2mf^*(H).f^*(K_{S/i})-3m-2-2m=-mf^*(H)f^*(K_{S/i})-5m-2$$
an it has to be greater than zero. By choosing as before $lH$ in place of $H$ and assuming that $f^*(lmH)-mE$ and $f^*(mH)-mE$ are both very ample, we have
$$-2mf^*(lH).f^*(K_{S/i})>2l(m+1)$$
and $$2l(m+1)>5m+2$$ for high values of $l$. Therefore in this case also the argument of \ref{cor1} works and we get that the involution acts as identity on $A_0(S)$.
\end{remark}

\begin{example}
Let $F$ be a singular quintic, invariant under an involution on $\PR^3$ and having simple elliptic singularities at the points
$$(1:0:0:0), (0:1:0:0), (0:0:1:0), (0:0:0:1)$$
as studied in \cite{DW}[section 2]. Let us consider the minimum desingularization of this surface $F$ and call it $V$. This surface $V$ is a smooth, projective surface of general type with $p_g=q=0$, equipped with an involution. The fixed locus of the involution on $F$ consists of a line and five isolated fixed points. These five points are different from the singular points of $F$. Let us consider the pre-images of these five points on $V$. They are the isolated fixed points of the involution on $V$. Consider the blow-up of $V$ at the five isolated fixed points of the involution on $V$. Denote it by $V'$. This surface $V'$ is equipped with an involution $i$. Then it is proven in \cite{DW}[proposition 3.1], that $V'/i$ is a non-singular, rational surface. So by the above remark, \ref{rem1}, the involution acts as identity on $A_0(V')$, provided that there exists a line bundle $L$ on $V'/i$ such that
$$L.(-K_{V'/i})-1>0\;.$$
Following the discussion in \cite{DW}[discussion after proposition 3.1] we consider the minimal model of $V'/i$. Call it $S$, it is a minimal elliptic surface as mentioned in \cite{DW}[discussion after proposition 3.1]. For this $S$ we have
$$K_S^2=0$$
then by Riemann-Roch
$$h^0(-K_S)\geq K_{S}^2+\chi (\bcO_S)+1=1+1=2$$
as $h^0(2K_S)=-1$ ($S$ is rational, so $|2K_S|=\emptyset$) and $\chi(\bcO_S)=1$. Therefore for a very ample line bundle of large degree on $S$, we have
$$-L.K_S-1>0\;.$$
Now by construction, as in \cite{DW}, the surface $S$ is a contraction of $V'/i$ along two elliptic curves of self-intersection $-1$. Let $\pi$ be the blow-down map from $V'/i$ to $S$. Therefore for a very ample line bundle $$L=\pi^*(L')-E_1-E_2$$ and $$K_{V'/i}=\pi^*(K_S)+E_1+E_2$$ on $V'/i$, we have
$$-(\pi^*(L')-E_1-E_2)(\pi^*(K_S)+E_1+E_2)-1=-\pi^*(L'.K_S)-3>0$$
for some very ample line bundle of the form $$L=m\pi^*(L')-E_1-E_2\;.$$ Here $m$ is a very large positive integer. Thus we have
$$-L.K_{V'/i}-1>0\;.$$
Therefore there exists a line bundle $L$ on $V'/i$ such that
$$L^2-2g+1=n$$
for some positive integer $n$, here $g$ is the genus of a smooth curve in $|L|$, as required in the condition of the corollary \ref{cor1}.
Since $V'/i$ is rational, the involution acts also as $-1$ resulting to the fact that every element in $A_0(V')$ is $2$-torsion and hence by Roitman's theorem $A_0(V')=\{0\}$ (as $q=0$ for $V'$). Since by the blow up formula
$$A_0(V)\cong A_0(V')$$
we have $A_0(V)=\{0\}$. Thus the Bloch's conjecture holds on $V$.
\end{example}

\subsection{Generalization of the above result}

The technique of the proof of \ref{theorem2} is more general, in the sense that we only use the conic bundle structure of the cubic fourfold and the conic bundle structure on the hyperplane sections of the cubic fourfold. Suppose that we consider a fourfold $X$, which is unirational, so contains sufficiently many lines. Now consider a fixed line $l$ on $X$, and project onto $\PR^3$ from this line. Suppose that the discriminant surface $S$ inside $\PR^3$ 
admits a double cover $T$ of $S$ branched along finitely many points, inside the Fano variety of lines $F(X)$ of $X$.

The proof of \ref{theorem2} tells us that we have  the following theorem:

\begin{theorem}
Let $X$ be a fourfold embedded in $\PR^5$, which admits a conic bundle structure. Let $S$ denote the discriminant surface for the conic bundle structure such that it admits a branched cover at finitely many points.  Then for any very ample line bundle $L$ on $S$, we cannot have the equality
$$L^2-g+1=g+n$$
where $g$ is the genus of a curve in the linear system of $L$ and $n$ is a positive integer.
\end{theorem}

\end{document}